\documentclass[12pt]{article}
\usepackage{graphicx}    
\usepackage{amssymb,amsmath,latexsym,amsbsy,amstext,amsthm}
\usepackage{amsthm}
\usepackage{hyperref}
\usepackage{caption}
\usepackage{subcaption}

\newtheorem{theorem}{Theorem}[section]
\newtheorem{lemma}[theorem]{Lemma}
\newtheorem{corollary}[theorem]{Corollary}

\newcommand{\al}{\alpha}
\newcommand{\be}{\beta}
\newcommand{\bexp}{{\sqrt{\frac{1+\be}{2}}}}

\newcommand{\dt}{\delta}
\newcommand{\e}{\varepsilon}

\newcommand{\g}{\gamma}

\newcommand{\I}{\mbox{Im\,}}

\newcommand{\ka}{\kappa}
\newcommand{\lbda}{\lambda}
\newcommand{\mb}{\mathbb}

\newcommand{\pr}{\mb{P}}
\newcommand{\R}{\mbox{Re\,}}

\newcommand{\ska}{\sqrt{\kappa}}

\newcommand{\vphi}{\varphi}

\begin{document}      
\title{Convergence of an algorithm simulating Loewner curves}         
\author{Huy Tran - University of Washington}
\date{\today}          
\maketitle
\begin{abstract}
The development of Schramm--Loewner evolution (SLE) as the scaling limits of discrete models from statistical physics makes direct simulation of SLE an important task. The most common method, suggested by Marshall and Rohde \cite{MR05}, is to sample Brownian motion at discrete times, interpolate appropriately in between and solve explicitly the Loewner equation with this approximation. This algorithm always produces piecewise smooth non self-intersecting curves whereas SLE$_\ka$ has been proven to be simple for $\ka\in[0,4]$, self-touching for $\ka\in(4,8)$ and space-filling for $\ka\geq 8$. In this paper we show that this sequence of curves converges to SLE$_\ka$ for all $\ka\neq 8$ by giving a condition on deterministic driving functions to ensure the sup-norm convergence of simulated curves when we use this algorithm. 
\end{abstract}  
\section{ Introduction}
	The Loewner equation uses a real valued function, the Loewner driving term, to describe a family of decreasing simply connected domains in the complex plane. It was first introduced by Charles Loewner as an attempt to solve the Bieberbach conjecture. This conjecture was completely solved by de Branges with the Loewner equation as one of the key tools \cite{Branges}. Oded Schramm rediscovered the Loewner equation when he was studying the scaling limits of discrete models. In this context, he introduced the Schramm--Loewner evolution (SLE$_\ka$, $\ka\geq 0$), a random growth process in the plane \cite{Sch00}. This process is obtained from the Loewner equation with a random driving term which is $\sqrt\ka$ times Brownian motion. By the work of Lawler, Schramm, Sheffield, Smirnov, Werner and others, SLE arises as a scaling limit of various discrete models from statistical physics \cite{LSW04}, \cite{SchrammSheffield05}, \cite{SchrammSheffield09}, \cite{Smirnov01}, \cite{Smirnov10}. It is therefore very desirable to generate pictures of SLE$_\ka$ directly to help understand the discrete random paths from those models.

 We are primarily interested in the case when the driving function corresponds to a growing curve. There are so far two methods to directly simulate the Loewner equation.  The first method uses the fact that the Loewner equation is a first order ODE, and hence one can numerically solve the equation, for example using Euler's method. Some of the first simulations of the SLE curve were obtained by Vincent Beffara using this method. It produces a good approximation to the SLE$_\ka$ hull (not the path) for $\ka>4$. One disadvantage is that it does not show the curve corresponding to the driving function but only a neighborhood of it, see Section \ref{sec:algorithm}. We note that SLE$_\ka$ is a random curve for all $\ka\geq 0$, see \cite{RS} and \cite{LSW04}. This approach has not been often used, and we do not discuss its convergence here.

 The second method for simulating SLE was suggested by Marshall and Rohde \cite{MR05}. The algorithm discretizes the driving function and square-root-interpolates it. As a result the algorithm approximates SLE maps by composing many basic conformal maps, which are easy to compute. The algorithm was described and implemented in \cite{K07}, \cite{K09} as well as in many other works. Oded Schramm was skeptical at first that the pictures generated from this algorithm well-present the SLE curves, according to Steffen Rohde \cite{R}. The curves simulated from the algorithm are piecewise smooth and simple, see Figure  \ref{fig:simulations}, whereas SLE$_\ka$ is a random fractal curve which is simple for $\ka\in[0,4]$, self-touching for $\ka\in(4,8)$ and plane-filling for $\ka\geq 8$. In this paper, we will prove:
\begin{theorem}
\label{theo:random}
For 
 $\ka\neq 8$, let $\g^n$ be the sequence of curves simulated from the second algorithm. Then under the half plane parametrization, the sequence $\g^n$  almost surely converges to SLE$_\ka$ in the sup-norm.
\end{theorem}
It is known that, for all $\ka$, the sequence $\g^n$ converges to SLE in the context of Carath\'eodory convergence \cite{L} and Cauchy transforms of probability measures \cite{Bauer}. However these types of convergence relate to Loewner chains rather than curves, see \cite[Chapter 4]{L} and respectively \cite{Bauer} for details. For $\ka\leq 4$, when one views curves as compact sets, the sequence $\g^n$ converges almost surely to SLE$_\ka$ in Hausdorff metric \cite[Section 7]{BJK}. A general principle is to set up a theorem for the deterministic Loewner equation and then translate the result into the SLE context. The theorem \ref{theo:random} will follow from a more general theorem for deterministic curves. In particular, we will show that there is a class of driving functions for which the sup-norm convergence of approximation curves occurs, see Theorem \ref{theo:deterministic} for the details of the statement. It is shown in \cite{MR05}, \cite{Lind} and \cite{LR} that driving functions whose H\"older-1/2 norms are less than 4 generate simple curves and that the Hilbert space filling curve is generated by a H\"older-1/2 function. Our Theorem \ref{theo:deterministic} is also applied to these driving functions.
\begin{corollary}\label{theo: Hilbert curve}
Consider the driving function that generates the Hilbert space filling curve or that has H\"older-1/2 norm less than 4. Then the sequence of curves simulated from the algorithm for this driving function converges uniformly.
\end{corollary}

We note that Theorem \ref{theo:deterministic} also provides the convergence rate of the simulation. The key is to estimate how the curve changes when we modify the driving function since the driving functions of simulated curves converge uniformly. There are two key estimates in the proof of Theorem \ref{theo:deterministic}. One is the boundary behavior of a conformal map and the other is the perturbation of a Loewner chain when there is a small change of its driving function. The latter is a Gronwall-type estimate which appears in \cite{J} and \cite{JRW}. The two estimates are both related to the growth of the derivative of conformal maps near the boundary which will provide to the assumptions of the main theorem \ref{theo:deterministic}.

The paper is organized as follows. In section \ref{sec:background}, we begin by reviewing the Loewner equation; we then state our main result and prove some preliminary lemmas. In section \ref{sec:algorithm}, we describe the first and second algorithms simulating the Loewner equation. Then the main theorem is proved in section \ref{sec:proof}. The applications will be discussed in Section \ref{sec:app} when we show the convergence rate and discuss several variants of the algorithm.

\emph{Acknowledgment.} The author would like to thank Steffen Rohde for numerous insightful conversations. The author also would like to thank  Elliot Paquette and Brent Werness for helpful comments on early drafts of this paper.
%
%
%
\begin{figure} [h]
	
	\vspace{-2.10 in}
	\hspace{-1.5 in}
	\begin{subfigure}[h]{0.43\textwidth}
                \includegraphics[width=5.9in]{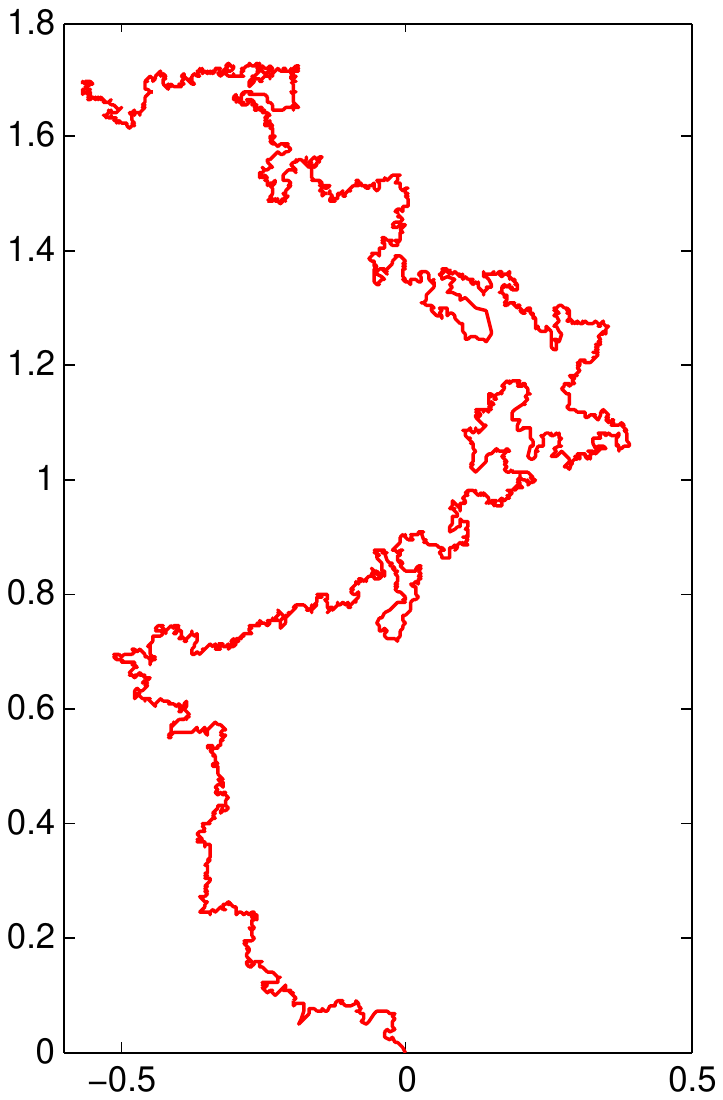}
            
        \end{subfigure}
	~~~~
        \begin{subfigure}[h]{0.1\textwidth}
             \includegraphics[width=5in]{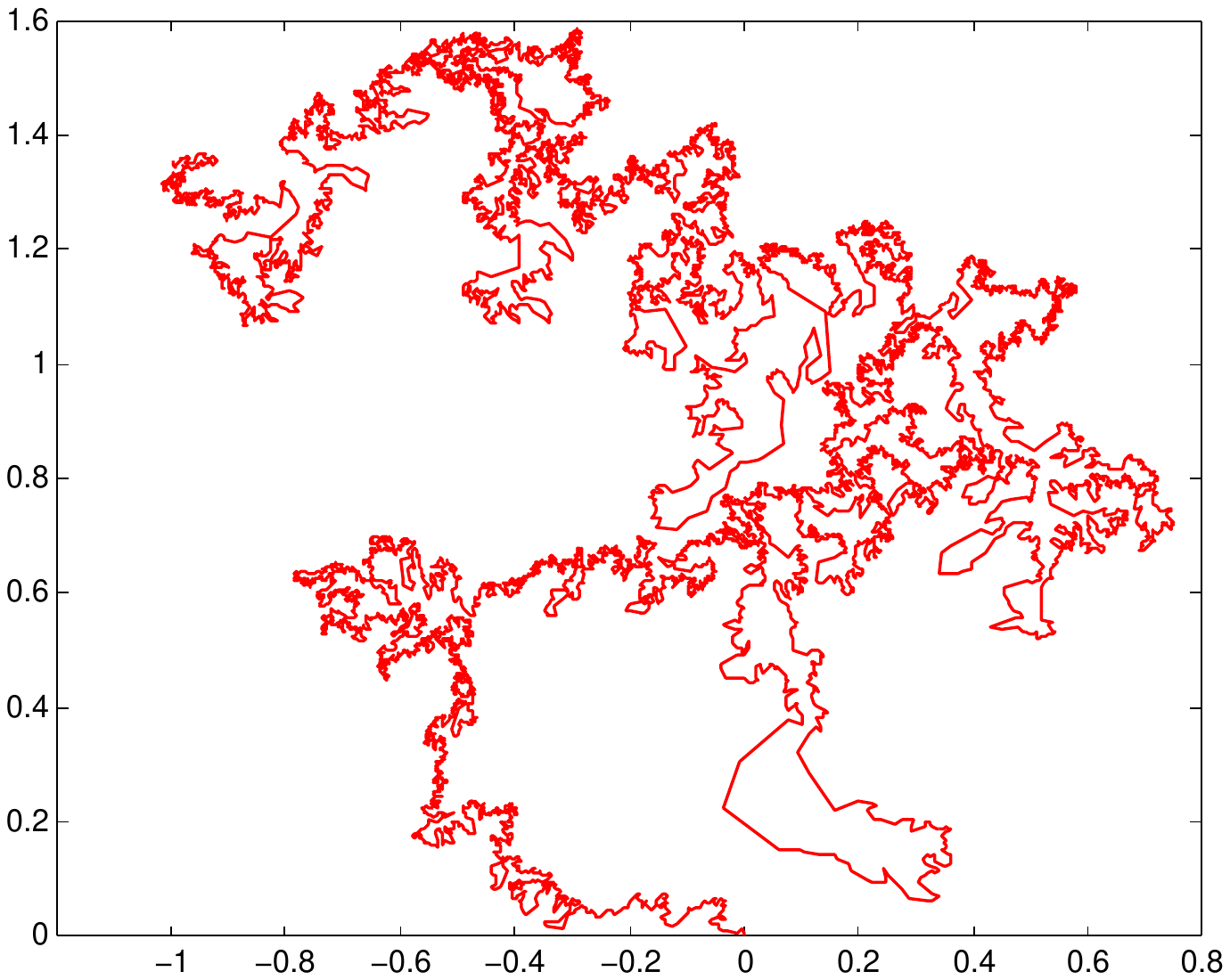}
        \end{subfigure}%
	
	\vspace{-2 in}
 	\caption{Simulations of SLE$_{8/3}$ (left) and SLE$_6$ (right) from the same Brownian motion sample with 12800 points}\label{fig:simulations}
\end{figure}
\section{ Loewner equation, algorithms, main result}
\label{sec:background}
\subsection{Loewner equation}
There are many versions of the Loewner equation. In this paper we focus on the (downward) chordal Loewner equation in the upper half plane $\mb{H}$:
\begin{equation}
\label{downward}
 \partial_t g_t(z)=\frac{2}{g_t(z)-\lbda(t)}
\end{equation}
with the initial condition $g_0(z)=z$ for every $z\in \mb{H}$, where $\lbda(t)$ is a real-valued continuous function defined for $t\geq 0$. Sometimes we write $\lbda_t$ for $\lbda(t)$. The family of $(g_t)_{t\geq 0}$ is called the \emph{Loewner chain}.

For each point $z\in \mb{H}$, the solution of (\ref{downward}): $t\mapsto g_t(z)$ is uniquely defined up to $T_z=\inf\{t\geq 0: g_t(z)=\lbda(t)\}$. As $t$ increases, the set $K_t=\{z\in\mb{H}:T_z\leq t\}$, called the \emph{hull}, grows.  It is known that for each $t\geq 0$, $g_t$ is the unique conformal map from $H_t:=\mb{H}\backslash K_t\to \mb{H}$ satisfying the \emph{hydrodynamic normalization} at $\infty$,
$$\lim_{z\to \infty} [g_t(z)-z] = 0.$$
It is usually easier to work with the \emph{upward Loewner equation}:
\begin{equation} 
\label{upward}
\partial_t h_t(z)=\frac{-2}{h_t(z)-\xi(t)},{\hskip 0.5in} h_0(z)=z,
\end{equation}
for $z\in\mb{H}$ and real-valued continuous function $\xi(t)$.

It is not hard to show that if $g_t$, $0\leq t\leq T$ is the solution to (\ref{downward}) with a driving function $\lbda$ and $h_t$ is the solution to (\ref{upward}) with $\xi(t)=\lbda(T-t)$ then
\begin{equation*}
h_T(z)=g_T^{-1}(z).
\end{equation*}

We are interested in the case that the Loewner chain $(g_t)$ is \emph{generated by a curve} $\g$, i.e., $H_t$ is the unbounded component of $\mb{H}\backslash \g([0,t])$. It follows from Theorem 4.1 (\cite{RS}) that this is equivalent to the existence and the continuity in $t>0$ of
$$\beta(t):=\lim_{y\to 0^+} g_t^{-1}(\lbda(t)+iy).$$ 
\subsection{Algorithms simulating Loewner equations}
\label{sec:algorithm}
Let us briefly discuss the first algorithm mentioned in the introduction. This idea to simulate $K_t$ is to determine whether a point $z$ in the upper half plane $\mb{H}$ satisfies $T_z\leq t$. One cannot examine all the points in $\mb{H}$ so if $T_z\leq t$ one declares that a certain neighborhood of $z$ is in $K_t$. To calculate the blow-up time $T_z$ one needs to run equation (\ref{downward}) until $g_t(z)$ hits $\lbda(t)$. However, there is no general method to solve (\ref{downward}) with given driving function. There are a few cases one can solve explicitly, see \cite{KKN}. As a result, if $\g$ is the simple curve corresponding to $\lbda$ then the simulation of $\g([0,t])=K_t$ is a neighborhood of the actual $\g([0,t])$. This is often not a good way to visualize the curve $\g$.

We now discuss the second algorithm to simulate Loewner curves. It was first appeared in \cite{MR05}. The algorithm has also been described in \cite{K07}, \cite{K09}, where modifications and fast implementations are discussed. One advantage of this algorithm is that it always produces simple curves. For the rest of the paper, this is the algorithm we consider, unless otherwise stated.

	The algorithm is based on two observations. First, fix $s>0$, and let $(\tilde{g}_t)$ be the solution of the Loewner equation with driving function $\tilde{\lbda}(t)=\lambda(s+t), t\geq 0$. This solution can be obtained by $g_{s+t}\circ g_s^{-1}$. Indeed 
$$\partial_t g_{s+t}\circ g^{-1}_s(z)=\frac{2}{g_{s+t}\circ g^{-1}(z)-\lbda(s+t)} =\frac{2}{g_{s+t}\circ g^{-1}(z)-\tilde{\lbda}(t)},$$
and $g_s\circ g_s^{-1}(z)=z$. By the uniqueness of solution of the equation (\ref{downward}), $\tilde{g_t}(z)=g_{s+t}\circ g^{-1}_s(z)$. If we let $\widetilde{K}_t$ be the hull associated with $\tilde{g}_t$ then 
$$g_s(K_{s+t})=\widetilde{K}_t \mbox{~and~} K_{s+t}=K_s\cup g_s^{-1}(\widetilde{K}_t).$$
 So in order to compute $K_{s+t}$, one can compute $K_s$ and $g_s^{-1}$, by using the information of $\lambda$ on $[0,s]$, and compute $\widetilde{K}_t$ by using $\lambda$ on $[s,s+t]$.

The second observation is that when $\lbda$ is of the form $c\sqrt{t}+d$, for some real constants $c$ and $d$, one can solve for $K_t$ explicitly. In this case, $K_t$ is a segment in the upper half plane starting at $d\in\mb{R}$ that makes an angle $\alpha \pi$ with the positive real axis where 
$$\alpha=\frac{1}{2}-\frac{1}{2}\frac{c}{\sqrt{16+c^2}},$$
and $g_t^{-1}(z+\lbda(t))=(z+2\sqrt{t}\sqrt{\frac{\alpha}{1-\alpha}})^{1-\alpha}(z-2\sqrt{t}\sqrt{\frac{\alpha}{1-\alpha}})^{\alpha}+d$. See \cite{KKN} for a proof.

We now fix a step $n\geq 1$. Let $t_k=\frac{k}{n}$ for $0\leq k\leq n$. So $t_0=0,t_1,\cdots,t_n=1$ is a partition of $[0,1]$. We will solve the Loewner equation with driving functions $\lambda(t+t_k)$ for $0\leq t\leq \frac{1}{n}$. By the remarks above, one should approximate these driving functions by $c\sqrt{t}+d$ so that one can solve explicitly the Loewner equation. More specifically, we approximate $\lbda$ by $\lbda^n$ such that they attain the same values at $t_k$'s and that $\lbda^n$ is a scaling and translation of $\sqrt{t}$ on $[t_k,t_{k+1}]$ from $\lbda(t_k)$ to $\lbda(t_{k+1})$. Hence the function $\lbda^n$ is defined as follows:
$$ \lbda^n(t)=\sqrt{n}(\lbda(t_{k+1})-\lbda(t_k))\sqrt{t-t_k} + \lbda(t_k) \mbox { on } [t_k,t_{k+1}].$$ 
This driving function always produces a simple curve $\g^n:[0,1]\to \mb{H}\cup \{\lbda(0)\}$.

\begin{figure} [h]

\centering
\includegraphics[width=6in,height=4in]{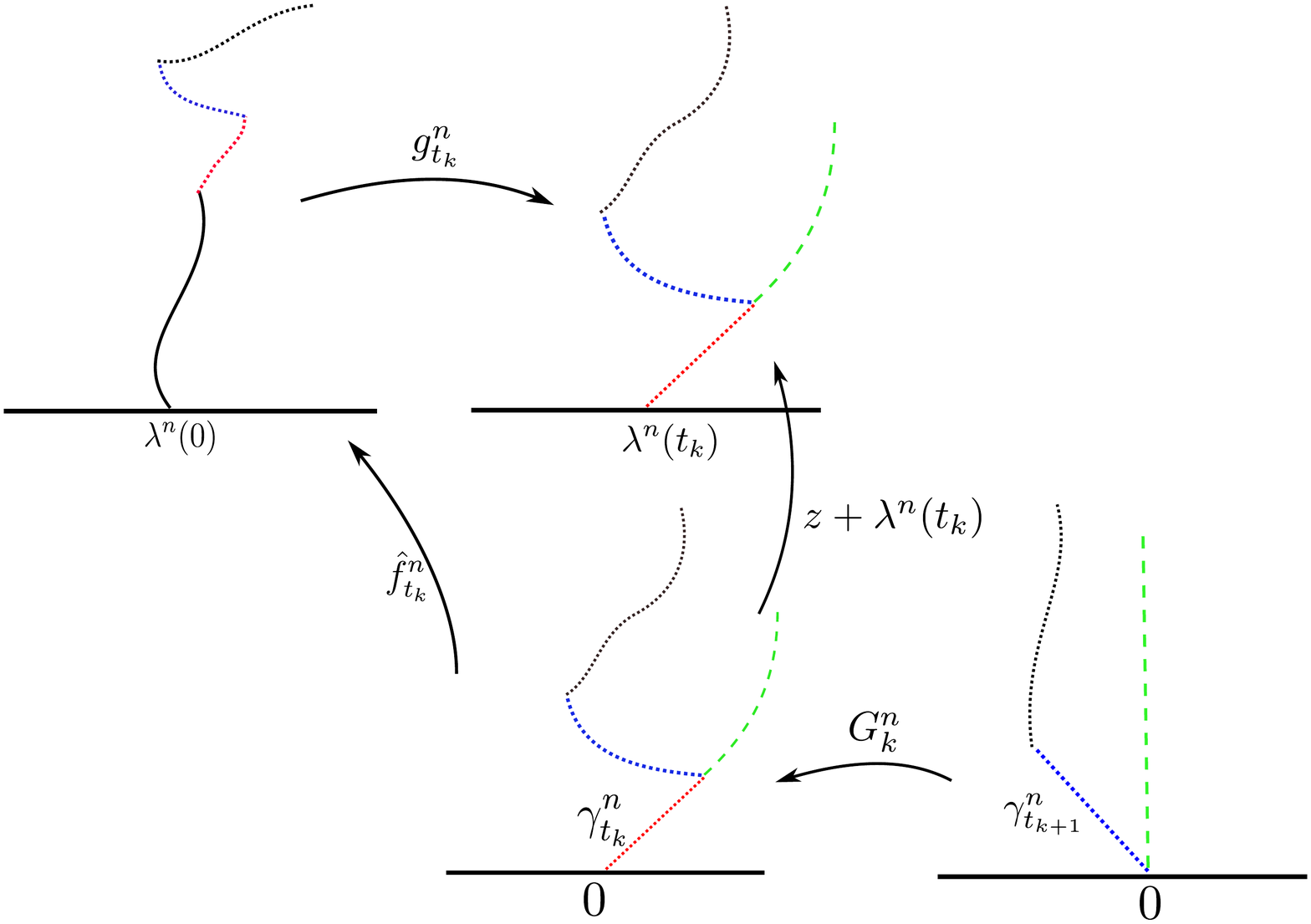}
\caption{At each step $k$, we compute $G^n_k$, $\hat{f}^n_{t_k}$ and $\g^n_{t_k}$. The $k$-th sub-arc of the simulation curve $\g^n$ is the image of $\g^n_{t_k}$ under $\hat{f}^n_{t_k}$.}
\label{figa}
\end{figure}
%
%

Denote by $(g^n_t)_{0\leq t\leq 1}$ the Loewner chain corresponding to $\lbda^n$. Let $f^n_t$ be the inverse function of $g^n_t$ and $\hat{f}^n_t(z)=f^n_t(z+\lbda^n(t))$. Define 
$$G^n_k=(\hat{f}^n_{t_k})^{-1}\circ \hat{f}^n_{t_{k+1}},$$
so that 
$$\hat{f}^n_{t_k}=G^n_{k-1}\circ G^n_{k-2}\circ\cdots\circ G^n_0.$$ 
For each $t\in [0,1]$, let $\g^n_t$ be the image of $\g^n$ under $g^n_t-\lbda^n(t)$, i.e., 
\begin{equation}
\g^n_t(s)=g^n_t(\g^n(t+s))-\lbda^n(t) \,\,\,\,\mbox{  and   }\,\,\, \g^n(t+s)=\hat{f}^n_t(\g^n_t(s))\,\,\, \mbox{ for } 0\leq s\leq 1-t.\end{equation}
We have chosen $\lbda^n$ so that $\g^n_{t_k}([0,\frac{1}{n}])$ is a segment starting at 0 and that $G^n_k$ has an explicit formula:
$$G^n_k(z)= \left(z+2\sqrt{\frac{1-\al}{n\al}}\right)^{1-\al}\left(z-2\sqrt{\frac{\al}{n(1-\al)}}\right)^\al$$
where $\al=\frac{1}{2}- \frac{1}{2}\frac{\sqrt{n}(\lbda(t_{k+1})-\lbda(t_k))}{\sqrt{16+ n(\lbda(t_{k+1})-\lbda(t_k))^2}}\in (0,1)$. See Figure \ref{figa}.

Therefore in order to compute $\g^n([0,1])$, we find $\g^n_{t_k}([0,\frac{1}{n}])$, $\hat{f}^n_{t_k}$, and then
\begin{equation}\label{form of g^n}
	\g^n(t) = \hat{f}^n_{t_k}(\g^n_{t_k}(t-t_k))\mbox{ for }  t\in [t_k,t_{k+1}), 0\leq k\leq n-1.
\end{equation}
Notice that $\lbda^n$ converges uniformly to $\lbda$ on $[0,1]$ since 
\begin{equation*}
	\sup_{t\in[0,1]}|\lbda^{n}(t)-\lbda(t)|\leq 2\sup_{s,t\in[0,1], |t-s|\leq \frac{1}{n}}|\lbda(t)-\lbda(s)|.	
\end{equation*}
We mention without proof a geometric property of $G^n_k$ which we will use later.

       \begin{lemma}
	\label{theo:G}
	Consider the conformal map $G(z)=(z+a)^{1-\al}(z-b)^{\al}$ from $\mb{H}$ to $\mb{H}$ minus a slit starting at 0, where $a,b>0$, $\al\in (0,1)$ and $\al a = (1-\al)b$. The point 0 is mapped to the tip of the slit. Then the imaginary part of $G(iy)$ is increasing on $(0,\infty)$. In particular, the image of $iy$ has a larger imaginary part than that of the tip of the slit.
	\end{lemma}
	The dashed line of Figure \ref{figa} illustrates this proposition.
%
%
%
%
\subsection{ Main results}
We shall consider driving functions which have the same regularity as Brownian motion. A \emph{subpower function} $\phi$ is a non decreasing function from $[0,\infty)$ to $[0,\infty)$ satisfying:
$$\lim_{x\to\infty} x^{-\nu}\phi(x)=0 \mbox{ for all } \nu>0.$$
If $\phi_1,\phi_2$ are subpower functions then so are $c\phi_1$, $\phi_1^c$ and $\max(\phi_1,\phi_2)$ for every $c>0$.

The function $\lbda$ is called \emph{weakly H\"older-1/2} if there exists a subpower function $\vphi$ such that
\begin{equation}
\label{wk}
\mbox{osc}(\lbda;\dt):=\sup\{|\lbda(t)-\lbda(s)|:s,t\in [0,1], |t-s|\leq \dt\}\leq \sqrt{\dt}\vphi(1/\dt)\mbox{~for all~~} \dt>0.
\end{equation}
It follows from P.Levy's theorem that the sample paths of Brownian motion are almost surely weakly H\"older-1/2 with subpower function $c\sqrt{\log(\dt)}$, $c>\sqrt{2}$, see \cite[Theorem I.2.7]{RY99}.

It is known that if $\lbda$ is weakly H\"older-1/2 and if there exist $c_0>0$, $y_0>0$ and $0<\be<1$ so that
\begin{equation}
\label{condition}
|\hat{f}'_t(iy)|\leq  c_0y^{-\be}\mbox{ for all }  0<y\leq y_0,\,\, t\in [0,1],
\end{equation}

where $\hat{f}_t(\cdot)=g_t^{-1}(\lbda(t)+\cdot)$, then $(g_t)_{0\leq t\leq 1}$ is generated by a curve, see \cite[section 3]{JL}. This is one of the main ideas to show the existence of SLE curves for $\ka\neq 8$ (\cite{RS}). We note that the Loewner chain of SLE$_8$ does not satisfy (\ref{condition}).

Our main theorem shows that under these hypotheses the algorithm gives the sup-norm convergence of the simulation curves.

\begin{theorem}
\label{theo:deterministic}
Suppose $\lbda$ is a weakly H\"older-1/2 driving function with a subpower function $\vphi$ and suppose the condition (\ref{condition}) is satisfied. Then the curve $\g$ generated from the Loewner equation can be approximated by the algorithm; that is, there exists a subpower function $\widetilde{\vphi}$ such that for all $n\geq \frac{1}{y_0^2}$ and $t\in [0,1]$,
\begin{equation}
\label{differ}
|\g^n(t)-\g(t)|\leq \frac{\widetilde{\vphi}(n)} {n^{\frac{1}{2}(1-\bexp)}},
\end{equation}
where $\g^n$ is the curve generated from the algorithm which is explained in section \ref{sec:algorithm}. The function $\widetilde{\vphi}$ depends on $\vphi, c_0$ and $\beta$.
\end{theorem}

A related question is the following: under what additional assumptions, does the uniform convergence of driving functions imply convergence of corresponding curves? {\it A priori} the convergence of curves occurs in the sense of Carath\'eodory convergence, see \cite{L}, and in the sense of Cauchy transform of probability measures, see \cite{Bauer} for definitions and details. As said in the introduction, these types of convergence do not directly involve to the curves. In \cite[Section 7]{BJK}, it is shown that if two driving functions generating simple curves are close in the sup-norm and if one function has the condition (\ref{condition}) then the two generated curves are close in Hausdorff distance. One really wants to see two curves are close in the sup-norm. Lind, Marshall and Rohde \cite{LMR} show that if the driving functions have H\"older-1/2 norm less than 4, then the curves converge uniformly. However, the Brownian motion is a.s. not H\"older-1/2. In \cite{SS}, the authors study sufficient conditions to have uniform convergence of bidirectional paths (the curves and their time-reversals). The paper by Johansson-Viklund \cite{J} uses the tip structure modulus to  get another criterion for uniform convergence of curves.

In the rest of this paper, $C$ stands for absolute constant and $\phi$ for general subpower functions; $c$ and $\vphi$ stand for constants and subpower functions that may depend on the assumptions of Theorem \ref{theo:deterministic}. They can change line by line and are indexed when necessary to avoid confusion.

Since we are interested in the same type of driving functions as those in \cite[section 3]{JL}, there are several results from their paper we will use and state here for the convenience of the reader. 
\begin{lemma}\label{theo:diam_hull}
\cite[Lemma 3.4]{JL} Let $K$ be a hull. There exists a constant $C<\infty$ such that 
$$\mbox{hcap}(K)\leq C\mbox{diam}(K) \mbox{height}(K).$$
\end{lemma} 

\begin{lemma}
\label{theo:oscillation}
\cite[Lemma 3.1]{C} Suppose $\g$ is the curve generated by a driving function $\lbda(t)$ in (\ref{downward}). Then for all $z\in \g([0,t])$,
$$ |\mbox{Re } z|\leq \sup_{0\leq s\leq r\leq t} |\lbda(r)-\lbda(s)|$$
and
$$\mbox{Im }z \leq 2\sqrt{t}.$$
\end{lemma}

\begin{lemma}\cite[Proposition 3.8]{JL} Let $(g_t)$ be the Loewner chain corresponding to $\lbda(t)$ satisfying (\ref{wk}) and (\ref{condition}). Then there exists a subpower function $\vphi_1$ such that if $0\leq t\leq t+s\leq 1$  and $s\in [0,y^2]$
$$|\g(t+s)-\g(t)|\leq \vphi_1(1/y)[v(t+s,y)+v(t,y)],$$
where $$v(t,y)=\int^y_0|\hat{f}'_t(ir)|dr\leq \frac{c_0}{1-\be}y^{1-\be},\,\,\,\, 0<y<y_0,$$
and that 
	\begin{equation}\label{uni cont}
		|\g(t+s)-\g(t)|\leq \vphi_1(1/y)\frac{2}	{1-\be}y^{1-\be}\mbox{ for } 0\leq s\leq y^2\leq y^2_0.
	\end{equation}
\end{lemma}
Most of the time we will deal with the behavior of conformal maps near the real and imaginary lines. For every subpower function $\phi$, constant $c>0$ and $n\in \mb{N}$, define
	$$A_{n,c,\phi}=\{x+iy\in \mb{H}: |x|\leq \frac{\phi(n)}{\sqrt n} \mbox{ and } \frac{1}{\sqrt{n}\phi(n)}\leq y\leq \frac{c}{\sqrt n}\}.$$

\begin{lemma}\label{theo:dilate}
	There exist constants $\al>0$ and $c'>0$ such that if $z_1$ and $z_2$ are inside the box $A_{n,c,\phi}$, 
	and $f$ is a conformal map on $\mb{H}$ then
		\begin{equation}
		\label{distortion of f}
		|f'(z_1)|\leq c' \phi(n)^\al |f'(i\I z_1)|
		\end{equation}
	and
		$$d_{\mb{H},hyp}(z_1,z_2)\leq c'\log \phi(n)+c'.$$
	The constants $\al$ and $c'$ depend only on $c$, not on $\phi$ or $n$. The notation $d_{H,hyp}(z_1,z_2)$ means the hyperbolic distance between $z_1$ and $z_2$ in a simply connected domain $H$.
\end{lemma}
\begin{proof}
The proof is similar to \cite[Lemma 3.2]{JL}.
\end{proof}
\begin{lemma}\label{theo:Pom}
\cite[Corollary 1.5]{Pommerenke} (half plane version) If $f$ is a conformal map of $\mb{H}$ into $\mb{C}$ and if $z_0,z_1\in\mb{H}$ then
$$
	 |f(z_1)-f(z_2)|\leq 2\, |(\I z_1)f'(z_1)|\exp(4d_{\mb{H},hyp}(z_1,z_2)).
$$
\end{lemma}
%
%
%
%
%
%

%
%
%
%

%
%
%
%
%
%
%
%

%
%
%
%

%
%
%
%
%
%
%
\section{Proof of main theorem \ref{theo:deterministic} }\label{sec:proof}
\subsection{ Heuristic argument}
	Let $\g_t$ be the image of $\g$ under $g_t-\lbda(t)$, i.e.
$$\g_t(s)=g_t(\g(t+s))-\lbda(t)\mbox { and } \, \g(t+s)=\hat{f}_t(\g_t(s)).$$
	We want to estimate 
	\begin{equation}
	\label{heuristic ineq}
	|\g(t_{k+1})-\g^n(t_{k+1})|=|\hat{f}_{t_k}(z)-\hat{f}^n_{t_k}(w)|
\leq |\hat{f}_{t_k}(z)-\hat{f}_{t_k}(w)|+|\hat{f}_{t_k}(w)-\hat{f}^n_{t_k}(w)|
	\end{equation}
where $z=\g_{t_k}(\frac{1}{n})$ and $w=\g^n_{t_k}(\frac{1}{n})$.

First, $z$ and $w$ are the tips of two curves generated respectively by two driving functions defined on $[0,\frac{1}{n}]$. It follows from Lemma \ref{theo:oscillation} that $\I z$ and $\I w \leq \frac{2}{\sqrt n}$.

For the first term in the RHS of (\ref{heuristic ineq}), it follows from Lemma \ref{theo:Pom} that
$$
	 |\hat{f}_{t_k}(z)-\hat{f}_{t_k}(w)|\leq 2\, |(\I z)\hat{f}'_{t_k}(z)|\exp(4d_{\mb{H},hyp}(z,w)).
$$
Notice that $z$ and $w$ trivially have positive imaginary parts. 

If we can show that $z$ and $w$ are in the same box $A_{n,c,\phi}$ then combining with a hypothesis of $\hat{f}'_{t_k}$ on $i\mb{R}^+$, we obtain similar inequalities for $|\hat{f}'_{t_k}(z)|$ and $|\hat{f}'_{t_k}(w)|$ from Lemma \ref{theo:dilate}. Then it follows that 
$$2\,\I z|\hat{f}'_{t_k}(z)|\exp(4d_{\mb{H},hyp}(z,w))\lesssim (\I z)^{1-\beta}\phi(n)\lesssim \frac{\phi(n)}{(\sqrt{n})^{1-\be}}\to 0 \mbox{ as } n\to \infty,$$
where the notation $f\lesssim g$ means that $f\leq Cg$ for some constant $C>0$.

Since $w$ is a tip of a straight line generated by a nice driving function, we can show $w$ is in a box $A_{n,c,\phi}$, see (\ref{eq:tip in box}). However, $z=\g_{t_k}(\frac{1}{n})$, in the case of SLE, has continuous density on the strip $\{x+iy: x\in \mb{R},0\leq y\leq 2\}$. So there might not exist a controllable-sized box $A_{n,c,\phi}$ that contains $z$. However Lemma \ref{theo:point s} shows the existence of a point in $\g_{t_k}([0,\frac{1}{n}]) \cap A_{n,c,\phi}$, and we will use this point instead of $\g_{t_k}(\frac{1}{n})$. Then  we use the uniform continuity of $\g$ to get back to $(\ref{heuristic ineq})$.

For the second term in the RHS of (\ref{heuristic ineq}), notice that
$$\hat{f}_{t_k}(w)-\hat{f}^n_{t_k}(w)=f_{t_k}(w+\lbda(t_k))-f^n_{t_k} (w+\lbda(t_k)).$$
This expression is a perturbation of two solutions of the upward Loewner equation (\ref{downward}) with two driving functions $t\mapsto \lbda(t_k-\cdot)$ and $t\mapsto\lbda^n(t_k-\cdot).$ 

We will use the following lemma from \cite{JRW} and the fact that $\lbda$ and $\lbda^n$ are close on $[0,t_k]$ and that $|{f}_{t_k}'(w+\lbda(t_k))|$ is well-controlled. 
\begin{lemma}
\label{theo:JRW}
\cite[Lemma 2.3]{JRW} 
Let $0<T<\infty$. Suppose that for $t\in[0,T]$, $f^{(1)}_t$ and $f^{(2)}_t$ satisfy the upward Loewner equation (\ref{upward}) with $W^{(1)}_t$ and $W^{(2)}_t$, respectively, as driving terms. Suppose that
$$\e=\sup_{s\in[0,T]}|W^{(1)}_s-W^{(2)}_s|.$$
Then if $u=x+iy\in \mb{H}$, then
$$|f_T^{(1)}(u)-f_T^{(2)}(u)|\leq \e\exp\left\{\frac{1}{2}\left[\log\frac{I_{T,y}|(f^{(1)}_T)'(u)|}{y}  \log\frac{I_{T,y}|(f^{(2)}_T)'(u)|}{y}\right]^{1/2}+\log\log\frac{I_{T,y}}{y}\right\},$$
where $I_{T,y}=\sqrt{4T+y^2}$.
\end{lemma}

Thus if $|(f^{(1)}_T)'(u)|\leq cy^{-\be}$, then 
$$|f_T^{(1)}(u)-f_T^{(2)}(u)|\lesssim \e y^{-\sqrt{(1+\be)/2}}.$$
If one can show furthermore
$$\e\leq \frac{\phi(n)}{\sqrt n}\mbox{~~ and ~~~} y=\I u= \I w \geq \frac{1}{\phi(n)\sqrt n}$$
then
$$ |f_T^{(1)}(u)-f_T^{(2)}(u)|\lesssim  \frac{\phi(n)^{c''}}{n^{\frac{1}{2}(1-\bexp)}}\to 0.$$
	From here, we only have an estimate for $|\g(t)-\g^n(t)|$ when $t=t_k$. To have an estimate on the whole interval we notice that 
$$\g^n([t_{k+1},t_{k+2}])=\hat{f}^n_{t_k}(\g^n_{t_k}[\frac{1}{n},\frac{2}{n}])=G^n_k(\g^n_{t_{k+1}}[0,\frac{1}{n}]).$$

It follows from a property of $G^n_k$ (Lemma \ref{theo:G}), that every point in $\g^n_{t_k}([\frac{1}{n},\frac{2}{n}])$ is in a box $A_{n,c,\phi}$ and hence we  can apply the same argument for (\ref{heuristic ineq}) with $\g^n(t_{k+1})$ being replaced by any $\g^n(t)$, $t_{k+1}\leq t\leq t_{k+2}$.

	Now we will go into the details of the proof.
\subsection{Proof of Theorem \ref{theo:deterministic}}
\label{sec:body of proof}
	Fix an arbitrary interval $I=[t_k,t_{k+2}]$, $0\leq k\leq n-2$. Denote $\g_k=\g_{t_k}$, $\g^n_k=\g^n_{t_k}$.

	We will estimate $|\g(s+t_k)-\g^n(r+t_k)|$ for all $r\in [\frac{1}{n},\frac{2}{n}]$ and with a specific $s$ chosen later. Combining with the uniform continuity of $\g$, we will have an estimate for 
$$|\g(r+t_k)-\g^n(r+t_k)|\,\,\,\mbox { with all } r\in[\frac{1}{n},\frac{2}{n}].$$

From now on, we will choose $n$ so that $\frac{1}{n}\leq y^2_0$. Denote $z=\g_k(s), w =\g^n_k(r)$. By the triangle inequality,
\begin{equation}\label{main ineq}
|\g(s+t_k)-\g^n(r+t_k)|=|\hat{f}_{t_k}(z)-\hat{f}^n_{t_k}(w)|
\leq |\hat{f}_{t_k}(z)-\hat{f}_{t_k}(w)|+|\hat{f}_{t_k}(w)-\hat{f}^n_{t_k}(w)|.
\end{equation}

{\bf The first term in the RHS of (\ref{main ineq}).} It follows from Lemma \ref{theo:Pom} that
\begin{equation}\label{first term}
	 |\hat{f}_{t_k}(z)-\hat{f}_{t_k}(w)|\leq (2\I z)|\hat{f}'_{t_k}(z)|\exp(4d_{\mb{H},hyp}(z,w)).
\end{equation}

	
The next lemma shows the existence of a point in $\g_k([0,\frac{2}{n}])\cap A_{n,c,\phi}$. 
\begin{lemma}\label{theo:point s}
	There exists a subpower function $\phi$ depending only on $\vphi$, $c_0$ and $\beta$ of Theorem \ref{theo:deterministic} such that for $n\geq 1$ and $0\leq k\leq n-1$, there exists $s\in [0,\frac{2}{n}]$ such that $\g_k(s)\in A_{n,2\sqrt{2},\phi}$.
\end{lemma}
\begin{proof}
	Since $\eta:=\g_k([0,\frac{2}{n}])$ is the curve generated by the Loewner equation (\ref{downward}) with driving function $\lbda(t_k+.)-\lbda(t_k)$ on $[0,\frac{2}{n}]$ and since $\lbda$ is weakly H\"older-1/2, it follows from Lemma \ref{theo:oscillation} that 
$$|\R \g_k(s)|\leq \sqrt{\frac{2}{n}}\vphi(\frac{n}{2})=:\frac{\vphi_2(n)}{\sqrt{n}}\,\,\, \mbox{ and }\,\,\,\, \I\g_k(s)\leq \frac{2\sqrt{2}}{\sqrt n}$$
for all $s\in [0,\frac{2}{n}]$.

	This implies
		$$\mbox{diam}(\eta)\leq   \mbox{height}(\eta)+\mbox{width}(\eta)  \leq \frac{2\sqrt{2}}{\sqrt n}+ 	\frac{2\vphi_2(n)}{\sqrt n}=\frac{\vphi_3(n)}{\sqrt{n}}.$$
It follows from Lemma \ref{theo:diam_hull} that 
		$$\frac{2}{n}=\mbox{hcap} (\eta) \leq C\mbox{ diam}(\eta)\mbox{ height} (\eta),$$
so, 
$$ \mbox{height} (\eta)\geq \frac{1}{\sqrt {n}\vphi_4(n)}.$$
The lemma follows by choosing a highest point and letting $\phi:=\vphi_5:=\max(\vphi_4,\vphi_2)$.
\end{proof}
With this specific point $\g_k(s)$, we can use the inequality (\ref{distortion of f}) in Lemma \ref{theo:dilate}. To have a bound for $\exp(4d_{\mb{H},hyp}(z,w))$ one needs to show that $w=\g^n_k(r)$ is also in $A_{n,2\sqrt{2},\vphi_5}$.
By the same argument in the above lemma, since $\g^n_k([0,\frac{1}{n}])$ and $\g^n_{k+1}([0,\frac{1}{n}])$ are line segments,
\begin{equation}
\label{eq:tip in box}
\g^n_k(\frac{1}{n})\mbox{ and } \g^n_{k+1}(\frac{1}{n})\in A_{n,2\sqrt{2},\vphi_5}.
\end{equation}
%
%
%
%

	\begin{lemma}
	\label{theo:w in box}
	There exists a subpower function $\phi$ depending only on $\vphi$, $c_0$ and $\be$ such that for all $n$, $k$ and $r\in [\frac{1}{n},\frac{2}{n}]$, $\g^n_k(r)$ is in the box $A_{n,2\sqrt{2},\phi}$.

	\end{lemma}

	\begin{proof}
	Notice that $\g^n_k(r)$ is the tip of the Loewner curve generated by the driving function $t\mapsto \lbda^n(t+t_k)-\lbda^n(t_k)$, $t\in [0,r]$. It follows from Lemma \ref{theo:oscillation} that 
\begin{equation}
	|\R \g^n_k(r)|\leq \sup\{ |\lbda^n(t+t_k)-\lbda^n(t_k)|,t\in [0,r]\}\leq \sqrt{\frac{2}{n}}\vphi(\frac{n}{2})= \frac{\vphi_2(n)}{\sqrt n}.
\end{equation}
and
	$$\I \g^n_k(r)\leq 2\sqrt{r}\leq 2\sqrt{\frac{2}{n}}.$$

	Now the rest of the proof is to find a lower bound for $\I \g^n_k(r)$. Fix $r\in [\frac{1}{n},\frac{2}{n}]$. Let $x+iy:=\g^n_{k+1}(r-\frac{1}{n})$, $G:=(\hat{f}^n_{t_k})^{-1}\circ \hat{f}^n_{t_{k+1}}$.
	
	Since $\g^n_{k+1}([0,\frac{1}{n}])$ is a line segment with the tip $\g^n_{k+1}(1/n)$ in $A_{n,2\sqrt{2},\vphi_5}$, by Lemma \ref{theo:dilate},
	$$d_{\mb{H}, hyp}(x+iy, iy)\leq C\log \vphi_5(n)+C$$
	and
	$$ d_{\mb{H}, hyp}(x+iy,iy)=d_{\mb{H}\backslash \g^n_k[0,\frac{1}{n}], hyp}(\g^n_k(r),G(iy))\geq d_{\mb{H},hyp} (\g^n_k(r), G(iy)).$$

	It follows from (\ref{eq:tip in box}) and Lemma \ref{theo:G} that
	$$\I \g^n_k(r)\geq  \frac{\I G(iy)}{C\vphi_5(n)^C}\geq \frac{\I \g^n_k(1/n)}{C\vphi_5(n)^C}\geq \frac{1}{C\sqrt{n} \vphi_5(n)^{C+1}}=\frac{1}{\sqrt{n}\vphi_6(n)}.$$
	So $\g^n_k(r)$ and $\g_k(s)$ are both in $A_{n,2\sqrt{2},\vphi_7}$ with $\vphi_7=\max(\vphi_5,\vphi_6)$.
	\end{proof}

	We now apply Lemma \ref{theo:dilate} and obtain
\begin{eqnarray}
 |\hat{f}_{t_k}(z)-\hat{f}_{t_k}(w)|&\leq &(2\I z)|\hat{f}'_{t_k}(z)|\exp(4d_{\mb{H},hyp}(z,w))\notag\\
&\leq & C(\I z)\vphi_7(n)^{\al}|\hat{f}'_{t_k}(i\I z)|\exp (C\log \vphi_7(n)+C)\notag \\
&\leq &(\I z)^{1-\be}Cc_0\vphi_7(n)^{\al}\exp (C\log \vphi_7(n)+C)\notag\\
&\leq & (\frac{2\sqrt{2}}{\sqrt n})^{1-\be}Cc_0\vphi_7(n)^{\al}\exp (C\log \vphi_7(n)+C)= \frac{\vphi_8(n)}{{\sqrt n}^{1-\be}}.
\label{hyper}
\end{eqnarray}
%
%
%
%
%

%
%
%
%
%
%
{\bf The second term in the RHS of (\ref{main ineq}).}\\
	Let $u=x+iy:= w+\lbda(t_k)$. Since $\lbda(t_k)=\lbda^n(t_k)$,
	$$\hat{f}_{t_k}(w)-\hat{f}^n_{t_k}(w)=f_{t_k}(u)-f^n_{t_k}(u).$$
	Applying Lemma \ref{theo:JRW}, we get
	$$|f_{t_k}(u)-f^n_{t_k}(u)|\leq \e\exp\left\{\frac{1}{2}\left[\log\frac{I_{t_k,y}|f'_{t_k}(u)|}{y}  \log\frac{I_{t_k,y}|(f^n_{t_k})'(u)|}{y}\right]^{1/2}+\log\log\frac{I_{t_k,y}}{y}\right\},$$
where $I_{t_k,y}=\sqrt{4t_k+y^2}$ and $\e=\sup_{t\in [0,t_k]}|\lbda(t)-\lbda^n(t)|\leq \frac{2\vphi(n)}{\sqrt n}$.

Since $y=\I u = \I w \in [\frac{1}{\sqrt n\vphi_7(n)},\frac{2\sqrt{2}}{\sqrt n}]$,
	$$\frac{I_{t_k,y}}{y}\leq 2\sqrt{2}\sqrt{n}\vphi_7(n).$$
Since $f'_{t_k}(u)=\hat{f}'_{t_k}(w)$,	
			$$|f'_{t_k}(u)|\leq \frac{c_0}{y^\be}\leq c_0 \vphi_7(n)^\be\sqrt{n}^\be.$$
Also
	$$|(f^n_{t_k})'(u)|\leq C(y^{-1}+1)\leq 2C\vphi_7(n)\sqrt{n},$$
where the first inequality comes holds for all hydrodynamic normalized conformal map of $\mb{H}$.

It follows that
	\begin{eqnarray}
|\hat{f}_{t_k}(w)-\hat{f}^n_{t_k}(w)|&\leq &\frac{2\vphi(n)}{\sqrt n}\exp\left\{\sqrt{\frac{1+\be}{2}}\log(c\vphi_7(n)\sqrt{n})+\log\log2\sqrt{2n}\vphi_7(n)\right\}\notag\\
&=:&\frac{\vphi_9(n)}{\sqrt{n}^{1-\bexp}}.
\label{cara}
	\end{eqnarray}
{\bf End of the proof of Theorem \ref{theo:deterministic}.}
 
It follows from (\ref{main ineq}), (\ref{hyper}) and (\ref{cara}) that
$$|\g(s+t_k)-\g^n(r+t_k)|\leq \frac{\vphi_8(n)}{{\sqrt n}^{1-\be}}+\frac{\vphi_9(n)}{(\sqrt{n})^{1-\bexp}} =: \frac{\vphi_{10}(n)}{(\sqrt{n})^{1-\bexp}}$$
for all $r\in [\frac{1}{n},\frac{2}{n}]$.

Using the uniform continuity (\ref{uni cont}) of $\g$, we obtain
$$|\g(r)-\g^n(r)|\leq \frac{\vphi_{11}(n)}{(\sqrt{n})^{1-\bexp}}$$
for all $r\in [t_{k+1},t_{k+2}]$ and $0\leq k\leq n-2$, hence for all $r \in [0,1]$.\qed

{\bf Proof of Corollary \ref{theo: Hilbert curve}.}\\
It was shown in \cite{MR05}, \cite{Lind} and \cite{LR} that the unbounded complements of the hulls generated by the driving function are John domains. Therefore, it follows from \cite[Chapter 5]{Pommerenke} that the condition (\ref{condition}) is satisfied.\qed
%
%
%
%
\section{Applications}\label{sec:app}
\subsection{Variants of the algorithm}
{\bf Variant 1.} 	The conclusion of Theorem \ref{theo:deterministic} still holds for every $\lbda^n$ that satisfies:
	\begin{equation}
	|\lambda^n(t_k)-\lambda(t_k)|\leq \frac{\vphi(n)}{\sqrt n}
	\end{equation}
and
	\begin{equation}
	\lambda^n(t)=\sqrt{n}(\lbda^n(t_k)-\lbda^n(t_k))\sqrt{t-t_k} + \lbda^n(t_k) \mbox { on } [t_k,t_{k+1}].
	\end{equation}

	Indeed, the main inequality (\ref{main ineq}) has a slightly change
$$
	|\g(s+t_k)-\g^n(r+t_k)|=|\hat{f}_{t_k}(z)-\hat{f}^n_{t_k}(w)|
\leq |\hat{f}_{t_k}(z)-\hat{f}_{t_k}(w+\lbda^n(t_k)-\lbda(t_k))|$$
\begin{equation}
+|{f}_{t_k}(w+\lbda^n(t_k))-f^n_{t_k}(w+\lbda^n(t_k))|.
\end{equation}

	We can see that $z$ and $w+\lbda^n(t_k)-\lbda(t_k)$ are still in the same box $A_{n,c,\phi}$. Hence the same argument follows.
\\

{\bf Variant 2.}
	Lemma \ref{theo:G} and the property (\ref{eq:tip in box}), hence Lemma \ref{theo:w in box}, are still true if instead of square-root-interpolating $\lbda^n$ on $[t_k,t_{k+1}]$, we consider any function interpolating between $\lbda^n(t_k)$ and $\lbda^n(t_{k+1})$ such that when we run the Loewner equation for $0\leq t\leq \frac{1}{n}$, the resulting $\g^n_{t_k}$ has non decreasing imaginary part on $[0,\frac{1}{n}]$.

	In particular, linear-interpolating $\lbda^n$
		$$\lbda^n(t)=\lbda^n(t_k)+n(\lbda^n(t_{k+1})-\lbda^n(t_k)) (t-t_k)\mbox { on } [t_k,t_{k+1}]$$
also give the same conclusion as in Theorem \ref{theo:deterministic} (see \cite{KKN} for linear driving functions).
\\

{\bf Variant 3.} Instead of using tilted slits on each small interval one can use vertical slits \cite{K09}. In this case, $\lbda^n$ is a step function:
	$$\lbda^n(t)=\lbda(t_k) \mbox{ for } t\in [t_k,t_{k+1})$$
	and 
	$$ \g^n_{t_k}(t) = \lbda(t_k)+2i\sqrt{t}\mbox { on } [0,\frac{1}{n}).$$
	However $\g^n$ defined by (\ref{form of g^n}) is not a curve. We can do as follows. We compute $\g^n$ at discrete points $t=t_0, t_1, \cdots, t_n$. Then connect them with a straight line in that order. As plotted in \cite{K09}, it is almost impossible to distinguish this curve and the one from the (main) algorithm. Indeed, the same proof of Theorem \ref{theo:deterministic} is carried over for this algorithm and the same conclusion of this theorem holds. 
\subsection{Speed of convergence to $SLE_\ka$}

We can estimate the speed of convergence of the algorithm to $\g^\ka:=$ SLE$_\ka$ with $\ka\neq 8$.
\begin{corollary}\label{appl-sle}
	There exist constants $c_1,c_2,c_3,c_4>0$ depending on $\ka$ such that
	$$\pr\left(||\g^\ka-\g^m||_{[0,1],\infty}\leq \frac{c_1(\log m)^{c_2}}{\sqrt{m}^{1-\bexp}}\mbox{ for all } m\geq n\right)\geq1-\frac{c_3}{n^{c_4}}.$$
\end{corollary}

	In other words, this corollary implies Theorem \ref{theo:random}. 

	We will apply Theorem \ref{theo:deterministic} to the case $\lbda(t)=\ska B_t$. It follows from \cite[Theorem 3.2.4]{LL} that there exist constants $c_1$ (depending on $\ka$) and $c_2$ such that
	\begin{equation}
	\label{sle1}
	\pr\left\{osc(\lbda;\frac{1}{m})\geq c_1\sqrt{\frac{\log m}{m}}\mbox{ for all } m\geq n \right\}\leq \frac{c_2}{n^2}.
	\end{equation}

	Notice that in Theorem \ref{theo:deterministic}, if $\vphi(n)=\sqrt{\log n}$ in (\ref{wk}) then goes through the proof, the subpower function in (\ref{differ}) is of the form $c(\log n)^{c'}$ for some constants $c$ and $c'$.

	It follows from \cite[Proposition 4.2]{JL} that there exist constants $\be'\in (0,1)$, $c_3$ and $c_4>0$ depending on $\ka$ such that
	$$\sum^{\infty}_{m=n}\sum^{2^{2m}}_{j=1}\pr\left\{|\hat{f}'_{(j-1)2^{-2m}}(i2^{-m})|\geq 2^{\be' m}\right\}\leq \frac{c_3}{2^{nc_4}}.$$

	This implies that
	$$\pr\left\{|\hat{f}'_{(j-1)2^{-2m}}(i2^{-m})|\leq 2^{\be' m}\mbox{ for all } 1\leq j\leq 2^{2m}, m\geq n \right\}\geq 1-\frac{c_3}{2^{n c_4}}.$$
	
	Thus there exist $c_5>0$ and $\be\in(\be',1)$ such that
	$$
	\pr\left\{|\hat{f}'_t(iy)|\leq \frac{c_5}{y^\be}\mbox{ for all } 0\leq y \leq 2^{-n}, t\in [0,1] \right\}\geq 1-\frac{c_3}{2^{nc_4}},
	$$
	or
	\begin{equation}
	\label{sle2}
	\pr\left\{|\hat{f}'_t(iy)|\leq \frac{c_5}{y^\be}\mbox{ for all } 0\leq y \leq \frac{1}{\sqrt n}, t\in [0,1] \right\}\geq 1-\frac{c_3}{n^{c_4/2}}.
	\end{equation}

	Combining (\ref{sle1}), (\ref{sle2}) and Theorem \ref{theo:deterministic}, we get
		$$\pr\left\{ ||\g^\ka-\g^m||_{[0,1],\infty}\leq \frac{c_6(\log m)^{c_7}}{\sqrt{m}^{1-\bexp}}\mbox{ for all } m\geq n \right\}\geq1-\left(\frac{c_2}{n^2}+\frac{c_3}{n^{c_4/2}}\right)$$
which proves Corollary \ref{appl-sle}. \qed
%
%
%
%

%
%
%
%
%
\subsection{Random walk algorithm to simulate SLE curves}
	This algorithm \cite[Section 2]{K09} is based on the Donsker's invariance theorem: a scaling limit of simple random walk converges in distribution to the Brownian motion.

	For fix $\ka\geq 0$. We choose $a\in (0,\frac{1}{2}]$ such that 
		$$\ka=\frac{4(1-2a)^2}{a(1-a)}.$$
	Let $f_1(z)=(z+1-a)^{1-a}(z-a)^a$, $f_2(z)=(z+a)^a(z-(1-a))^{1-a}$.
	For every $i\geq 1$, choose $\phi_i=f_1$ or $\phi_i=f_2$ with equal probability. Then we compute inductively $F_n =F_{n-1}\circ \phi_n$ with $F_0=id$. The map $F_n$ is conformal from $\mb{H}$ to $\mb{H}$ minus a slit curve. After rescaling and translating so that this slit curve has the half plane capacity 1, we get a simple curve $\g^n$. More explicitly, $\g^n$ is generated by $\lbda^n$ whose formula is
	$$\lbda^n(t_k)=\ska \frac{S_k}{\sqrt n} \mbox{ for all } t_k,$$
$$ \mbox{ and } \lbda^n(t)=\sqrt{n}(\lbda^n(t_{k+1})-\lbda^n(t_k))\sqrt{t-t_k} + \lbda^n(t_k) \mbox { on } [t_k,t_{k+1}],$$
where $S_k=X_1+\cdots+X_k$, $X_i's$ are iid and $\pr(X_i=1)=\pr(X_i=-1)=\frac{1}{2}$.

By Donsker's invariance theorem, $\lbda^n\stackrel{d}{\rightarrow} \ska B|_{[0,1]}$ on $C([0,1],||.||_\infty)$ . So $\mb{H}\backslash\g^n([0,1])\stackrel{d}{\rightarrow} \mb{H}\backslash \g^\ka ([0,1])$ in the context of Carath\'eodory kernel convergence \cite{L} and Cauchy transforms of probability measures \cite{Bauer}. Kennedy \cite{K09} raised a question whether $\g^n$ converges in distribution to $\g^\ka$.

We now show that $\g^n$ converges in distribution to $\g^\ka$ under the sup-norm of $C([0,1])$ when $\ka\neq 8$. Indeed, it follows from \cite[Theorem 7.1.1]{LL} that for each $n$, we can couple $\lbda^n$ and the Brownian motion in the same probability space such that
\begin{equation}
\label{sle3}
\pr\{\max_{0\leq j\leq n}{|\lbda^n(t_j)-\ska B_{t_j}|}\geq \frac{C \ska\log n}{\sqrt n}\}\leq C n^{-3}
\end{equation}
for some universal constant $C>0$.

Hence from $(\ref{sle1}), (\ref{sle2}), (\ref{sle3})$ and the discussion of Variant 1, there exist constants $c_8$ and $c_9$ depending on $\ka$ such that
$$\pr\left\{||\g^n-\g^\ka||_{[0,1],\infty}\leq\frac{c_8(\log n)^{c_9}}{\sqrt{n}^{1-\bexp}}       \right\}\geq 1-\frac{c_2}{n^2}-\frac{c_3}{n^{c_4/2}}-\frac{C}{n^3}.$$
This implies that $\g^n$ converges in distribution to $\g^\ka$.


\bibliographystyle{alpha}
\bibliography{SLE}
\end{document}